\documentclass[12pt]{article}
\usepackage{amsmath,amsfonts,amssymb,amsthm}
\usepackage{bbm,mathrsfs,bm,mathtools}
\title{A weighted Minkowski theorem for pseudo-cones}
\author{Rolf Schneider}
\date{}
\newcommand{\Sn}{{\mathbb S}^{n-1}}

\newcommand{\R}{{\mathbb R}}

\newcommand{\C}{{\mathcal C}}

\newcommand{\K}{{\mathcal K}}

\newcommand{\N}{{\mathbb N}}

\newcommand{\Ha}{\mathcal{H}}

\newcommand{\B}{\mathcal{B}}
\newcommand{\D}{{\rm d}}

\newtheorem{theorem}{Theorem}
\newtheorem{lemma}{Lemma}
\newtheorem{definition}{Definition}

\begin{document}
\maketitle

\begin{abstract}
A nonempty closed convex set in $\R^n$, not containing the origin, is called a pseudo-cone if with every $x$ it also contains $\lambda x$ for $x\ge 1$. We consider pseudo-cones with a given recession cone $C$, called $C$-pseudo-cones. The family of $C$-pseudo-cones can, with reasonable justification, be considered as a counterpart to the family of convex bodies containing the origin in the interior. For a $C$-pseudo-cone one can naturally define a surface area measure and a covolume. Since they are in general infinite, we introduce a weighting, leading to modified versions of surface area and covolume. These are finite and still homogeneous, though of other degrees. Our main result is a Minkowski type existence theorem for $C$-pseudo-cones with given weighted surface area measure. \\[2mm]
\end{abstract}

\section{Introduction}\label{sec1}

A pseudo-cone in $\R^n$ is a nonempty closed convex set $K$ with the property that $o\notin K$ (where $o$ denotes the origin of $\R^n$) and that $\lambda K\subseteq K$ for all $\lambda\ge 1$. Like convex bodies containing the origin in the interior, pseudo-cones admit a kind of polarity, now called copolarity, which is defined by
$$ K^*:= \{x\in\R^n:\langle x,y\rangle\le -1\mbox{ for all }y\in K\},$$
where $\langle\cdot\,,\cdot\rangle$ denotes the scalar product of $\R^n$. If $K$ is a pseudo-cone, then its copolar set $K^*$ is again a pseudo-cone, and $K^{**}=K$. Copolarity (though not under this name) was pointed out in \cite{ASW22} as a special example within a systematic study of dualities and was investigated in \cite{XLL22}, and pseudo-cones were treated in \cite{Ras17, Sch23}.

If $K$ is a pseudo-cone and ${\rm rec}\,K$ is its recession cone, then $K\subset{\rm rec}\,K$, and ${\rm rec}\,K^*= ({\rm rec}\, K)^\circ$, where $C^\circ$ denotes the 
 dual cone of a convex cone $C$. In the following, we will deal with pseudo-cones with given recession cone $C$, which we call $C$-pseudo-cones. Of the cone $C$, we always assume that it is pointed and $n$-dimensional. If $C$ is given, the subsets 
$$ \Omega_C:= \Sn\cap {\rm int}\,C,\qquad \Omega_{C^\circ}:= \Sn\cap{\rm int}\,C^\circ$$
of the unit sphere $\Sn$ (where ${\rm int}$ denotes the interior) will play an important role. 

With their copolarity, pseudo-cones are an interesting counterpart to convex bodies containing the origin in the interior. Many notions familiar for convex bodies can be carried over to pseudo-cones. We mention first the surface area measure. Let $K$ be a $C$-pseudo-cone. Its surface area measure $S_{n-1}(K,\cdot)$ is the image measure of the $(n-1)$-dimensional Hausdorff measure, restricted to $\partial K\cap{\rm int}\,C$, under a suitably defined spherical image map. Here $\partial K$ denotes the boundary of $K$, and the spherical image map is only defined at boundary points of $K$ where there exists a unique outer unit normal vector, and this is in $\Omega_{C^\circ}$. The surface area measure $S_{n-1}(K,\cdot)$ is a Borel measure, but contrary to the case of convex bodies, it is only defined on $\Omega_{C^\circ}$, and it may be infinite. The latter fact poses new problems, in particular if one asks for an analogue of Minkowski's existence and uniqueness theorem.

For convex bodies, Minkowski's problem asks for necessary and sufficient conditions on a Borel measure on the unit sphere to be the surface area measure of a convex body, and for the degree of uniqueness of such a body. The answers are well known. However, necessary and sufficient conditions on a Borel measure on $\Omega_{C^\circ}$ to be the surface area measure of a $C$-pseudo-cone are not known. It was shown in \cite{Sch21} that every nonzero finite Borel measure on $\Omega_{C^\circ}$ is the surface area measure of a $C$-pseudo-cone. This pseudo-cone $K$ has the additional property of being $C$-close, which means that $C\setminus K$ has finite volume (and among $C$-close pseudo-cones, $K$ is unique). For finite measures, Minkowski type theorems for pseudo-cones were also investigated in the $L_p$ Brunn--Minkowski theory, by Yang, Ye and Zhu \cite{YYZ23}, and for dual curvature measures by Li, Ye and Zhu \cite{LYZ23}. Much less progress has been made for infinite measures. In \cite{Sch23}, a sufficient condition on a not necessarily finite Borel measure on $\Omega_{C^\circ}$ to be the surface area measure of a $C$-close pseudo-cone was found. It remained open whether this condition is also necessary. For a necessary condition satisfied by surface area measures of general $C$-pseudo-cones, we refer to \cite[Sect. 4]{Sch21}.

Unlike convex bodies, pseudo-cones have the property that far away from $o$ their shape is rather restricted: if the distance from $o$ tends to infinity, a $C$-pseudo-cone looks more and more like its recession cone $C$. Thus, close to the origin the surface area measure of a pseudo-cone has higher influence on the shape of the pseudo-cone and should, therefore, be given higher weight. For convex bodies, weighted Minkowski problems have been considered before, see Livshyts \cite{Liv19} and Krysonov and Langharst \cite{KL23}. For $C$-pseudo-cones, a weighting seems particularly natural, since it allows to neglect those regions where the surface area measure has decreasing influence on the shape of the pseudo-cone.

Therefore, we define the following class of functions. The simplifying assumptions are appropriate for our purposes.

\begin{definition}\label{D1.1} 
Let $\Theta: C\setminus\{o\}\to(0,\infty)$ be continuous and homogeneous of degree $-q$, where $q\in\R$.
\end{definition}

An example is given by $\Theta(y):= \|y\|^{-q}$ for $y\in C\setminus\{o\}$. To give another example, we recall from \cite[Sect. 2]{Sch23} that we can choose a unit vector ${\mathfrak v}\in{\rm int}\,C$ with $-{\mathfrak v}\in{\rm int}\, C^\circ$; this vector will be fixed in the following. Then $\Theta(y):= \langle y,\mathfrak v\rangle^{-q}$, $y\in C\setminus\{o\}$, also satisfies the assumptions.

Let $K$ be a $C$-pseudo-cone, and write $\partial_i K:= \partial K\cap {\rm int}\,C$. Let $\partial{\hspace{1pt}}'K$ be the set of boundary points $y\in\partial_i K$ where the outer unit normal vector $\nu_K(y)$ is unique and belongs to $\Omega_{C^\circ}$. (Possibly $\partial{\hspace{1pt}}'K=\emptyset$, namely if $K= C+z$ with $z\in C\setminus\{o\}$.) It is known that the set $\sigma_K$ of boundary points of $K$ where the outer unit normal vector is not unique satisfies $\Ha^{n-1}(\sigma_K)=0$ (where $\Ha^k$ denotes the $k$-dimensional Hausdorff measure) and that $\nu_K:\partial K\setminus\sigma_K \to {\rm cl}\,\Omega_{C^\circ}$ is continuous (${\rm cl}$ denotes the closure). Therefore, it makes sense to define
$$ S_{n-1}^\Theta(K,\omega):= \int_{\nu_K^{-1}(\omega)} \Theta(y) \,\Ha^{n-1}(\D y)$$
for Borel sets $\omega\subset\Omega_{C^\circ}$. This yields a Borel measure on $\Omega_{C^\circ}$, which we call the {\em $\Theta$-weighted surface area measure} of $K$. For convex bodies, surface area measures with respect to a given measure were introduced in \cite{Liv19}. Weighted surface area measures appear also in \cite{FLMZ23} and \cite{LPRY23}.

We shall prove the following weighted Minkowski theorem for pseudo-cones. 

\begin{theorem}\label{T1.1}
Let $\Theta$ be according to Definition $\ref{D1.1}$, and suppose that $n-1<q<n$. Then $S_{n-1}^{\Theta}(K,\cdot)$ is finite for every $C$-pseudo-cone $K$.

Let $\varphi$ be a nonzero, finite Borel measure on $\Omega_{C^\circ}$. Then there exists a $C$-pseudo-cone $K$ with
$$ S_{n-1}^{\Theta}(K,\cdot) =\varphi.$$
\end{theorem}

In general, $K$ is not unique up to translation, as we will explain after the proof of Theorem \ref{T1.1}.

In the next section, we fix some notation and collect elementary facts about pseudo-cones. Section \ref{sec3} introduces weighted surface area measures and proves a weak continuity property. Then weighted covolumes are considered. Section \ref{sec5} deals with Wulff shapes, here yielding pseudo-cones. This serves as a preparation for Section \ref{sec6}, where a variational lemma is proved. In principle, and for convex bodies, this goes back to Aleksandrov. However, his proof, involving Minkowski's inequalities for mixed volumes, cannot be used, and we have to adapt to pseudo-cones an approach found by Huang, Lutwak, Yang and Zhang \cite{HLYZ16}. Then follows the proof of Theorem \ref{T1.1}.

\section{Preliminaries on pseudo-cones}\label{sec2}

We have already said that we work in $\R^n$ ($n\ge 2$) and that $\Sn$ is the unit sphere of $\R^n$. Further, $B^n$ is the unit ball. By $\B(\cdot)$ we denote the $\sigma$-algebra of Borel sets of a given space.

We repeat that a pointed, $n$-dimensional closed convex cone $C\subset\R^n$ is given. We recall that 
$$ C^\circ=\{x\in\R^n:\langle x,y \rangle\le 0\mbox{ for all }y \in C\} $$ 
is its dual cone. The recession cone of a nonempty closed convex set $A\subseteq\R^n$ can be defined by
$${\rm rec}\,A= \{z\in\R^n: A+z\subseteq A\}$$
(cf. \cite[§8]{Roc70}). The set of all $C$-pseudo-cones, that is, pseudo-cones with recession cone $C$, is denoted by $ps(C)$.

We use the unit vector $\mathfrak v\in({\rm int}\,C)\cap(-{\rm int}\,C^\circ)$ to define, for $t>0$, the sets
$$ C(t):= \{y\in C: \langle y,\mathfrak v\rangle =t\},\quad C^-(t):= \{y\in C: \langle y,\mathfrak v\rangle \le t\}.$$
These sets are compact.

Convergence of $C$-pseudo-cones is defined as in \cite{Sch23}:

\begin{definition}\label{D2.0}
If $K_j$, $j\in\N_0$, are $C$-pseudo-cones, we write
$$ K_j\to K_0$$
and say that $(K_j)_{j\in\N}$ converges to $K_0$ if there exists $t_0>0$ such that  $K_j\cap C^-(t_0)\not=\emptyset$ for all $j\in \N_0$ and
$$ \lim_{j\to\infty} (K_j\cap C^-(t)) = K_0\cap C^-(t)\quad\mbox{for all }t\ge t_0,$$
where the latter means convergence of convex bodies with respect to the Hausdorff metric $d_H$.
\end{definition}

We mention the following counterpart to the Blaschke selection theorem (see \cite[Lem. 1]{Sch23}).

\begin{lemma}\label{L2.1}
Every sequence of $C$-pseudo-cones for which the distances from the origin are bounded and bounded away from zero has a subsequence that converges to a $C$-pseudo-cone.
\end{lemma}

In the following, we try to keep the notation close to that of \cite{HLYZ16}. There are differences, due to the non-compactness of pseudo-cones. One has also to be careful with parts of a $C$-pseudo-cone contained in the boundary of $C$. Moreover, the role of the unit sphere $\Sn$ in \cite{HLYZ16} is now split between its two subsets $\Omega_{C}$ and $\Omega_{C^\circ}$, which are not closed. For greater clarity, we try to reserve the letter $u$ for vectors in $\Omega_{C^\circ}$ and the letter $v$ for vectors in $\Omega_C$. Similarly, we often denote subsets of $\Omega_{C^\circ}$ by $\omega$ and subsets of $\Omega_C$ by $\eta$.

Let $K\in ps(C)$. The {\em support function} of $K$ is defined by
$$ h_K(x) := \sup\{\langle x,y\rangle: y\in K\}\quad\mbox{for } x\in C^\circ.$$
The supremum is an attained maximum if $x\in{\rm int}\,C^\circ$. For $x\in\partial C^\circ$, it may either be attained or not attained. We have restricted the domain to $C^\circ$ to obtain a real-valued support function. The support function is continuous and sublinear. Since for pseudo-cones it is non-positive, we often write $\overline h_K:= - h_K$.

For $u\in\Omega_{C^\circ}$ we denote by
$$ H_K(u):= \{y\in \R^n: \langle y,u\rangle = h_K(u)\}, \quad H^-_K(u):= \{y\in \R^n: \langle y,u\rangle \le h_K(u)\}$$ 
the {\em supporting hyperplane} of $K$ with outer unit normal vector $u$ and the halfspace containing $K$ that it bounds.

The {\em radial function} of $K$ is defined by 
$$ \varrho_K(y) := \min\{r\in\R: ry\in K\}\quad\mbox{ for }y\in C_K,$$
where
$$ C_K:= \{y\in C: ry\in K\mbox{ for some }r>0\} \supseteq {\rm int}\,C.$$
The latter inclusion holds since $C$ is the recession cone of $K$. We point out that $\varrho_K(y)y\in\partial K$. Note that $C_K$ is the maximal domain on which $\varrho_K$ can be defined as a real number. We set also
$$ \Omega_{C,K}:= C_K\cap \Sn.$$
The radial function is continuous, positive, and homogeneous of degree $-1$. 

For $u\in{\rm cl}\,\Omega_{C^\circ}$,
\begin{eqnarray*}
h_K(u) &=& \sup\{\langle u,y\rangle:y\in K\} = \sup\{\langle u,y\rangle:y\in \partial K\} \\
&=& \sup\{\langle u,\varrho_K(v)v\rangle: v\in\Omega_{C,K}\} = \sup\{\langle u,\varrho_K(v)v\rangle: v\in\Omega_C\}
\end{eqnarray*}
(where the last equality holds since $\Omega_{C,K}\subseteq{\rm cl}\,\Omega_C$), hence
\begin{equation}\label{2.2}
h_K(u) = \sup_{v\in\Omega_C} \langle u,v\rangle \varrho_K(v)\quad\mbox{for } u\in{\rm cl}\,\Omega_{C^\circ}.
\end{equation}

For $v\in\Omega_C$,
\begin{eqnarray*}
\frac{1}{\varrho_K(v)} &=& \max\{ 1/r:rv\in K\} = \max\{1/r: \langle rv,u\rangle\le h_K(u)\,\forall\,u\in{\rm cl}\,\Omega_{C^\circ}\}\\
&=& \max\left\{1/r: \frac{|\langle v,u\rangle|}{\overline h_K(u)}\ge\frac{1}{r} \,\forall\,u\in{\rm cl}\,\Omega_{C^\circ}\right\}\\
&=& \max \left\{\lambda : \min_{u\in{\rm cl}\,\Omega_{C^\circ}} \frac{|\langle v,u\rangle|}{\overline h_K(u)} \ge\lambda\right\},
\end{eqnarray*}
hence
\begin{equation}\label{2.0}
\frac{1}{\varrho_K(v)} = \min_{u\in{\rm cl}\,\Omega_{C^\circ}} \frac{|\langle v,u\rangle|}{\overline h_K(u)} \quad\mbox{for }v\in\Omega_C.
\end{equation}

\noindent{\bf Remark.} 
It follows from the derivation that the minimum is attained at $u\in {\rm cl}\,\Omega_{C^\circ}$ if and only if $u$ is an outer unit normal vector of $K$ at $\varrho_K(v)v$.

The following three definitions can essentially be carried over from the case of convex bodies. As mentioned, the notation is kept close to that of \cite{HLYZ16}. We always assume that $K$ is a given $C$-pseudo-cone.

\begin{definition}\label{D2.1}
The reverse spherical image of $\omega\subseteq \Omega_{C^\circ}$ (with respect to $K$) is defined by
$$ \bm x_K(\omega):=\{y\in\partial K: y\in H_K(u) \mbox{ for some }u\in\omega\}.$$
\end{definition}

Recall that $\sigma_K$ is the set of all $y\in\partial K$ where the outer unit normal vector is not unique. Let $\omega_K$ be the set of all $u\in\Omega_{C^\circ}$ for which the set $\bm x_K(\{u\})$ contains more than one element. 

\begin{definition}\label{D2.2}
For $y\in \partial{\hspace{1pt}}' K$, let $\nu_K(y)\in \Omega_{C^\circ}$ be the unique outer unit normal vector of $K$ at $y$. The map $\nu_K:\partial{\hspace{1pt}}'K\to\Omega_{C^\circ}$ is called the {\em spherical image map} of $K$.

For $u\in \Omega_{C^\circ}\setminus \omega_K$, let $x_K\in \partial K$ be the unique element of $\bm x_K(\{u\})$. The map $x_K:\Omega_{C^\circ}\setminus \omega_K\to\partial K$ is called the {\em reverse spherical image map} of $K$.
\end{definition}

The following assertions follow directly from the corresponding ones for convex bodies (see \cite[p. 339]{HLYZ16} for references). As already mentioned, we have $\Ha^{n-1}(\sigma_K)=0$, and the spherical image map $\nu_K$ is continuous on $\partial{\hspace{1pt}}' K\setminus \sigma_K$. We have $\Ha^{n-1}(\omega_K)=0$, and the reverse spherical image map $x_K$ is continuous on $\Omega_{C^\circ}\setminus \omega_K$. 

\begin{definition}\label{D2.3}
The {\em radial map} $r_K: \Omega_{C,K}\to\partial K$ of $K$ is defined by
$$ r_K(v) := \varrho_K(v)v\quad\mbox{for }v\in\Omega_{C,K}.$$
We define $\eta_K:= r_K^{-1}(\sigma_K)$.

The {\em radial Gauss map} $\alpha_K:\Omega_{C,K}\setminus\eta_K\to {\rm cl}\,\Omega_{C^\circ}$ of $K$ is defined by
$$ \alpha_K:= \nu_K\circ r_K.$$
\end{definition}

We recall from \cite[Sect. 8]{Sch18} that $K$ is {\em $C$-determined} by the nonempty, compact set $\omega\subset\Omega_{C^\circ}$ if
$$ K= C\cap\bigcap_{u\in\omega} H^-_K(u).$$
By $\K(C,\omega)$ we denote the set of $C$-pseudo-cones that are $C$-determined by $\omega$. If $K\in\K(C,\omega)$, then the absolute support function $\overline h_K$, restricted to $\omega$, is positive and bounded away from zero. If $K$ is $C$-determined by $\omega$, then $C\setminus K$ is bounded, since $\omega$ has a positive spherical distance from the boundary of $\Omega_{C^\circ}$. A $C$-pseudo-cone $K$ is called $C$-determined if it is $C$-determined by some nonempty, compact set $\omega\subset\Omega_{C^\circ}$. 

Given a pseudo-cone $K\in ps(C)$ and a nonempty, compact set $\omega\subset\Omega_{C^\circ}$, we can define a $C$-determined set by
\begin{equation}\label{2.1}
 K^{(\omega)}:= C\cap\bigcap_{u\in\omega} H^-_K(u).
\end{equation}
For this construction, we shall need the following lemma.

\begin{lemma}\label{L2.2}
If $K_j\in ps(C)$, $j\in\N_0$, satisfy $K_j\to K_0$ as $j\to\infty$ and if $\emptyset\not=\omega\subset\Omega_{C^\circ}$ is compact, then
$$ K_j^{(\omega)} \to K_0^{(\omega)}\quad\mbox{as } j\to\infty.$$
\end{lemma}

\begin{proof}
First we note that the convergence $K_j\to K_0$ implies the existence of a number $t_0$ such that $K_j\cap C^-(t_0)\not=\emptyset$ for $j\in \N_0$ and $K_j\cap C^-(t)\to K\cap C^-(t)$ for $t\ge t_0$, with convergence in the sense of the Hausdorff metric for convex bodies. Since $\omega$ is a compact subset of $\Omega_{C^\circ}$, we can choose $t_0$ so large that $\bm x_{K_j}(\omega)\subset C^-(t_0)$ for all $j$. In particular, $h_{K_j}(u)\to h_{K_0}(u)$ for $u\in\omega$.

We use Theorem 1.8.8 of \cite{Sch14}. Let $x\in K_0^{(\omega)}$. Let $x_j$ be the point where the ray through $o$ and $x$ intersects the boundary of $K_j^{(\omega)}$. Then $x_j$ is contained in the boundary of a halfspace $\{y\in\R^n:\langle y,u\rangle \le h_{K_j}(u)\}$ for some $u\in\omega$, and $x$ is contained in the halfspace 
$\{y\in\R^n:\langle y,u\rangle\le h_{K_0}(u)\}$. The boundaries of these halfspaces have distance $|h_{K_j}(u)-h_{K_0}(u)|\le d_H(K_j\cap C^-(t_0),K_0\cap C^-(t_0))$. The angle $\alpha$ between $x$ and $u$ satisfies $|\cos\alpha|\ge a$ with some constant $a>0$, since $\omega$ has a positive spherical distance from the boundary of $\Omega_{C^\circ}$. It follows that $\|x_j-x\|\le d_H(K_j\cap C^-(t_0),K_0\cap C^-(t_0))/a$. We conclude that $x=\lim_{j\to\infty} x_j$; thus (a) of \cite[Thm. 1.8.8]{Sch14} is satisfied.

Suppose that $x_j\in K_{i_j}^{(\omega)}$ for $j\in\N$ and $x_j\to x$. Then $x\in C$. For each $u\in\omega$ we have $\langle x_j,u\rangle \le  h_{K_{i_j}}(u)$ for all $j$, hence also $\langle x,u\rangle \le  h_{K_0}(u)$. This shows that $x\in K_0^{(\omega)}$. We have proved that also (b) of \cite[Thm. 1.8.8]{Sch14} is satisfied.
\end{proof}

\section{Weighted surface area measures}\label{sec3}

Since in the following we deal with measures on the sphere, we confirm first the convention that the unit sphere $\Sn$ is equipped with the usual spherical Lebesgue measure and integrals such as $\int_\beta \cdot\,\D u$, where $\beta\subset \Sn$ is Lebesgue measurable, always refer to integration with respect to this measure. Also `measurable' without further comment means `Lebesgue measurable', and `almost all $u\in\beta$' means all $u\in\beta$ up to a subset of $\beta$ of Lebesgue measure zero.

In the following, we assume that a function $\Theta$  according to Definition \ref{D1.1} is given. In this section, we assume that $q>n-1$. For a $C$-pseudo-cone $K$, we denote by $S_{n-1}^{\Theta}(K,\cdot)$ its $\Theta$-weighted surface area measure. 

\begin{lemma}\label{L3.1}
If $q>n-1$, the measure $S_{n-1}^{\Theta}(K,\cdot)$ is finite for every $C$-pseudo-cone $K$.
\end{lemma}

\begin{proof}
First, let $y\in C\setminus\{o\}$. There exists $\lambda>0$ with $\langle \lambda y,\mathfrak v\rangle =1$. Since $C(1)$ is compact and $\Theta$ is continuous on $C(1)$, there exists a number $c_0>0$ with $\Theta(z) \le c_0$ for $z\in C(1)$. Hence, $\Theta(\lambda y) \le c_0\langle \lambda y,\mathfrak v\rangle^{-q}$. Since $\Theta$ is homogeneous of degree $-q$, it follows that 
\begin{equation}\label{3.1}
\Theta(y)\le c_0\langle y,\mathfrak v\rangle^{-q} \quad \mbox{for }y\in C\setminus\{o\}.
\end{equation}

We choose $z \in {\rm int}\,K$, then the translated cone $C_z:= C+z$ is contained in $K$ (since $C$ is the recession cone of $K$). Further, we choose $t_0>0$ with $C(t_0)\cap {\rm int}\,C_z\not=\emptyset$. Since
$$ \int_{\partial K\cap C^-(t_0)} \Theta(y)\,\Ha^{n-1}(\D y)<\infty,$$
it suffices to show that
$$ \int_{\partial K\setminus C^-(t_0)}  \Theta(y)\,\Ha^{n-1}(\D y)<\infty.$$
To prove this, we want to use that the surface area of convex bodies is monotone under set inclusion. Therefore, we define $t_j:=t_0+j$ for $j\in\N$ and the convex bodies
$$ K_j:= \{y\in K: t_j\le \langle y,\mathfrak v\rangle\le t_{j+1}\},$$
together with parts of their boundary,
$$ S_j:= \{y\in\partial K_j: t_j< \langle y,\mathfrak v\rangle< t_{j+1}\}.$$
We compare $K_j$ with the cylinder
$$ Z_j:= \left[K\cap C(t_{j+1})\right]+\{\lambda \mathfrak v: -1\le \lambda\le 0\}$$
and its boundary part
$$ T_j:= \{y\in\partial Z_j: t_j< \langle y,\mathfrak v\rangle< t_{j+1}\}.$$  
We need also the $(n-2)$-dimensional sets
$$ m_j:= \partial K\cap C(t_j).$$
Since $K_j\subset Z_j$ we have
$$ \Ha^{n-1}(S_j)\le \Ha^{n-1}(T_j) + \Ha^{n-1}((T_j\setminus K_j)\cap C(t_j)).$$
Here we have 
$$\Ha^{n-1}(T_j) = \Ha^{n-2}(m_{j+1}),$$
and since $C_z\subset K$, we can estimate
$$ \Ha^{n-1}((T_j\setminus K_j)\cap C(t_j))\le c_1\Ha^{n-2}(m_j)\le c_1\Ha^{n-2}(m_{j+1})$$
with a constant $c_1$ independent of $j$. Thus,
\begin{equation}\label{3.2}
\Ha^{n-1}(S_j) \le c_2 \Ha^{n-2}(m_{j+1})
\end{equation}
with a constant $c_2$. Using (\ref{3.1}) and (\ref{3.2}), we obtain
\begin{eqnarray*}
&& \int_{\partial K\setminus C^-(t_0)}  \Theta(y)\,\Ha^{n-1}(\D y)\le c_0 \int_{\partial K\setminus C^-(t_0)} \langle y,\mathfrak v\rangle^{-q}\,\Ha^{n-1}(\D y)\\
&&= c_0 \sum_{j=0}^\infty\int_{S_j} \langle y,\mathfrak v\rangle^{-q}\,\Ha^{n-1}(\D y) \le c_0 \sum_{j=0}^\infty\int_{S_j} t_j^{-q}\,\Ha^{n-1}(\D y)\\
&&\le c_0c_2\sum_{j=0}^\infty t_j^{-q}\Ha^{n-2}(m_{j+1}) \le c_0c_2 \sum_{j=0}^\infty t_j^{-q}\Ha^{n-2}(\partial C\cap C(t_{j+1}))\\
&&= c_0c_2 \sum_{j=0}^\infty t_j^{-q} t_{j+1}^{n-2}\Ha^{n-2}(\partial C\cap C(1)) <\infty,
\end{eqnarray*}
as stated. It is now obvious why we have assumed that $q>n-1$. 
\end{proof}

We need a lemma to transform integrals over $\partial_i K$ to integrals over $\Omega_C$, via the radial map. For convex bodies, a corresponding lemma is proved in \cite[Lem. 2.9]{HLYZ16} by using the divergence theorem for sets of finite perimeter (\cite[§5.8]{EG92}), and also in \cite[p. 171]{HW20} by employing the coarea formula for rectifiable manifolds (\cite[Thm. 3.2.22]{Fed69}). The result appeared already as Lemma 1 in Aleksandrov \cite{Ale39}, but with a different interpretation of the surface area of a Borel set in the boundary of a convex body, not using the Hausdorff measure. The proof given in \cite{HW20}  works for pseudo-cones as well. Nevertheless, we want to indicate also an approach along more familiar lines.

\begin{lemma}\label{L2a.1}
Let $K$ be a $C$-pseudo-cone. Let $F:\partial_i K\to\R$ be nonnegative and  Borel measurable or $\Ha^{n-1}$-integrable. Then
\begin{eqnarray} 
\int_{\partial_i K} F(y)\,\Ha^{n-1}(\D y) =\int_{\Omega_C} F(r_K(v))\frac{\varrho_K^n(v)}{\overline h_K(\alpha_K(v))}\,\D v\label{2a.0}\\
 =\int_{\Omega_C} F(r_K(v))\frac{\varrho_K^{n-1}(v)}{|\langle v,\alpha_K(v)\rangle|}\,\D v.\label{2a.00}
\end{eqnarray}
\end{lemma}

\begin{proof}
We recall from \cite[Sect. 10]{Sch18} that the cone-volume measure of $K$ is defined by
$$ V_K(\omega):= \Ha^n\left(\bigcup_{y\in\bm x_K(\omega)} [o,y]\right)$$
for $\omega\in\B(\Omega_{C^\circ})$, where $[o,y]$ is the closed segment with endpoints $o$ and $y$, and that the use of spherical coordinates gives
\begin{equation}\label{3.1n}
V_K(\omega) =  \frac{1}{n}\int_{\alpha_K^{-1}(\omega)} \varrho_K^n(v)\,\D v
\end{equation}
(where $\alpha_K^{-1}(\omega):= \{v\in\Omega_C\setminus\eta_K: \alpha_K(v)\in\omega\}$). By \cite[Lem. 9]{Sch18} and the interpretation of $S_{n-1}(K,\cdot)$ as the image measure of $\Ha^{n-1}$, restricted to $\partial_i K$, under the spherical image map we have
\begin{equation}\label{3.2n}
 V_K(\omega) = \frac{1}{n} \int_{\omega} \overline h_K(u)\,S_{n-1}(K,\D u)= \frac{1}{n} \int_{\bm x_K(\omega)}|\langle y,\nu_K(y)\rangle|\,\Ha^{n-1}(\D y).
\end{equation}

Now let $\eta\in\B(\Omega_C)$. Relations (\ref{3.1n}) and (\ref{3.2n}), applied to $\omega = \alpha_K(\eta\setminus \eta_K)$ (so that $\bm x_K(\omega)= r_K(\eta)$, up to a set of $\Ha^{n-1}$-measure zero), show that 
$$ \int_{\partial_i K} {\mathbbm 1}_\eta(r_K^{-1}(y)) |\langle y,\nu_K(y)\rangle|\,\Ha^{n-1}(\D y) = \int_{\Omega_C} {\mathbbm 1}_\eta(v)\varrho_K^n(v)\,\D v.$$
Thus, the relation
$$ \int_{\partial_i K} g(r_K^{-1}(y)) |\langle y,\nu_K(y)\rangle|\,\Ha^{n-1}(\D y) = \int_{\Omega_C} g(v)\varrho_K^n(v)\,\D v$$
holds if $g$ is the characteristic function of a Borel set in $\Omega_C$. By the usual extension procedure for integrals, it holds if $g$ is any nonnegative measurable or integrable function on $\Omega_C$. With $g(v):= F(r_K(v))/\overline h_L(\alpha_K(v))$, we obtain (\ref{2a.0}), and the version (\ref{2a.00}) follows from $\overline h_K(\alpha_K(v))=|\langle \varrho_K(v)v,\alpha_K(v)\rangle|$.
\end{proof}

For expressing integrals with respect to the weighted surface area measure as integrals over parts of $\Omega_C$, we use the following lemma.

\begin{lemma}\label{L3.2n}
Let $\omega\subseteq\Omega_{C^\circ}$ be a nonempty Borel set and $g:\omega\to\R$ be bounded and measurable. Then
$$ \int_\omega g(u)\,S^\Theta_{n-1}(K,\D u) = \int_{\alpha_K^{-1}(\omega)} g(\alpha_K(v))\Theta(r_K(v))\frac{\varrho^{n-1}_K(v)}{|\langle v,\alpha_K(v)\rangle|}\,\D v.$$
\end{lemma}
\begin{proof}
With $F(y):= {\mathbbm 1}_\omega(\nu_K(y))g(\nu_K(y))\Theta(y)$, Lemma \ref{L2a.1} yields
$$ \int_{\partial_i K} {\mathbbm 1}_\omega(\nu_K(y))g(\nu_K(y))\Theta(y)\,\Ha^{n-1}(\D y) = \int_{\alpha_K^{-1}(\omega)} g(\alpha_K(v))\Theta(r_K(v)) \frac{\varrho_K^{n-1}(v)}{|\langle v,\alpha_K(v(\rangle|}\,\D v.$$

By the definition of $S^\Theta_{n-1}(K,\cdot)$ and the transformation formula for integrals, we have
$$ \int_\omega g(u)\,S^\Theta_{n-1}(K,\D u) = \int_{\partial_i K}{\mathbbm 1}_\omega(\nu_K(y)) g(\nu_K(y))\Theta(y)\,\Ha^{n-1}(\D y).$$
These two equations together give the assertion.
\end{proof}

The following lemma shows that weighted surface area measures are weakly continuous on each set $\K(C,\omega)$ with compact $\omega$.

\begin{lemma}\label{L3.3}
Let $\omega\subset\Omega_{C^\circ}$ be a nonempty, compact subset, and let $K_j\in\K(C,\omega)$ for $j\in\N_0$. Then $K_j\to K_0$ as $j\to\infty$ implies the weak convergence $S^\Theta_{n-1}(K_j,\cdot)\stackrel{w}{\to} S^\Theta_{n-1}(K_0,\cdot)$.
\end{lemma}

\begin{proof}
Let $g:\Omega_{C^\circ}\to\R$ be bounded and continuous. By Lemma \ref{L3.2n} we have
$$ \int_{\Omega_{C^\circ}} g(u)\,S^\Theta_{n-1}(K_j,\D u)= \int_{\Omega_C} g(\alpha_{K_j}(v))\Theta(r_{K_j}(v))\frac{\varrho_{K_j}^{n-1}(v)}{|\langle v,\alpha_{K_j}(v)\rangle|}\,\D v.$$
For almost all $v\in\Omega_C$ it holds that
$$ g(\alpha_{K_j}(v))\Theta(r_{K_j}(v))\frac{\varrho_{K_j}^{n-1}(v)}{|\langle v,\alpha_{K_j}(v)\rangle|} \to g(\alpha_{K_0}(v))\Theta(r_{K_0}(v))\frac{\varrho_{K_0}^{n-1}(v)}{|\langle v,\alpha_{K_0}(v)\rangle|}$$
as $j\to\infty$. The functions on the left-hand side are uniformly bounded, since $K_j\to K_0$ and $K_j$ is $C$-determined by the compact set $\omega\subset\Omega_{C^\circ}$. Therefore, $\|r_{K_j}(v)\|$ is bounded by a constant. Further, $|\langle v,\alpha_{K_j}(v)\rangle|$ is bounded away from $0$, since the spherical distance from the boundary of $\Omega_{C^\circ}$ is continuous and strictly positive on $\omega$. Now the dominated convergence theorem proves the assertion.
\end{proof}

\section{Weighted covolumes}\label{sec4}

Additionally to $n-1<q$, in this section we also assume that $q<n$. This assumption ensures that
\begin{equation}\label{4.1} 
\int_{B^n\cap C} \Theta(y)\,\Ha^n(\D y)<\infty.
\end{equation}
In fact, since for $y\in C$ we have $\Theta(y)\le c_0\langle y,\mathfrak v\rangle^{-q}$ by (\ref{3.1}) and $\langle y,\mathfrak v\rangle\ge c_1\|y\|$ for $y\in C$ with some constant $c_1$ (because of $-{\mathfrak v}\in{\rm int}\,C^\circ$), (\ref{4.1}) follows from
$$ \int_{B^n\cap C} \Theta(y)\,\Ha^n(\D y) \le c_0c_1^{-q} \int_{B^n} \|y\|^{-q}\,\Ha^n(\D y) = c_0c_1^{-q} \int_0^1 r^{-q} r^{n-1}\omega_n\,\D r<\infty,$$
where $\omega_n$ is the surface area of $\Sn$. By homogeneity and the positivity of $\Theta$ it follows that 
\begin{equation}\label{4.2} 
\int_{C^-(t)}\Theta(y)\,\Ha^n(\D y) <\infty
\end{equation}
for all $t>0$.

For a Borel subset $A\subset C$ we define 
\begin{equation}\label{4.2a} 
\Ha^n_\Theta(A) := \int_{A} \Theta(y)\,\Ha^n(\D y),
\end{equation}
and for a $C$-pseudo-cone $K$ we call
$$ V_\Theta(K):= \Ha^n_\Theta(C\setminus K)$$
the {\em $\Theta$-weighted covolume} of $K$. For general $\Theta$, this can be infinite.

\begin{lemma}\label{L4.1}
If $n-1<q<n$, the $\Theta$-weighted covolume is finite and continuous (in the sense that $K_j\to K$ implies $V_\Theta(K_j)\to V_\Theta(K)$).  
\end{lemma}

\begin{proof}
Let $K\in ps(C)$ be a $C$-pseudo-cone. We choose $t>0$ so that $z:=t\mathfrak v\in K$ and set $C_z:= C+z$. Since ${\mathfrak v}\in{\rm int}\,C$ and $C$ is the recession cone of $K$, the number $t$ exists and we have $C_z\subseteq K$. Because of (\ref{4.2}), it suffices to prove that $\Ha^n_\Theta((C\setminus K)\setminus C^-(t))<\infty$. Using (\ref{3.1}), we get
\begin{eqnarray*} 
\Ha^n_\Theta((C\setminus K)\setminus C^-(t)) &=& \int_t^\infty\int_{(C\setminus K)\cap C(\tau)} \Theta(y)\,\Ha^{n-1}(\D y)\,\D\tau\\
&\le& c_0 \int_t^\infty\int_{(C\setminus K)\cap C(\tau)} \tau^{-q}\,\Ha^{n-1}(\D y)\,\D\tau\\
&\le& c_0 \int_t^\infty \tau^{-q}\left[\Ha^{n-1}(C\cap C(\tau)) -\Ha^{n-1}(C_z\cap C(\tau))\right]\,\D\tau.
\end{eqnarray*}
Writing $\Ha^{n-1}(C\cap C(1))=:c_2$, we have $\Ha^{n-1}(C\cap C(\tau))=\tau^{n-1}c_2$ and $ \Ha^{n-1}(C_z\cap C(\tau)) = (\tau-t)^{n-1}c_2$, hence
$$ \Ha^n_\Theta((C\setminus K)\setminus C^-(t)) \le c_0c_2\int_t^\infty \tau^{-q} [\tau^{n-1}-(\tau-t)^{n-1}]\,\D\tau<\infty,$$
showing the finiteness of the $\Theta$-weighted covolume on $C$-pseudo-cones.

To prove the continuity, let $K, K_j$, $j\in\N$, be $C$-pseudo-cones with $K_j\to K$ as $j\to \infty$. We choose $t,z,C_z$ as above, but now such that $z\in{\rm int}\,K$. Let $\varepsilon>0$ be given. It follows from the argument above that there exists a number $s>0$ such that $z\in {\rm int}\,C^-(s)$ and 
\begin{equation}\label{4.3} 
\Ha^n_\Theta((C\setminus C_z)\setminus C^-(s))<\varepsilon/2.\end{equation}
Then we choose a number $s_0>0$ with $K\cap C^-(s_0)=\emptyset$ and a constant $b$ with
$$ \Theta(y)\le b\quad\mbox{for}\quad y\in C^-(s)\setminus C^-(s_0).$$

We remark that for convex bodies $A,B\subset C^-(s) $ with $d_H(A,B)\le\delta$ (for some $\delta>0$), we have $A\setminus B\subseteq (B+\delta B^n)\setminus B$ and hence $V_n(A\setminus B)\le V_n(B+\delta B^n)-V_n(B)$. Using the Steiner formula and the monotonicity of its coefficients, we obtain $V_n(A\setminus B)\le \delta a$ with a constant $a$ depending only on $C^-(s)$.

Now we choose $\delta>0$ such that, for $L\in ps(C)$, the inequality
\begin{equation}\label{4.4}
d_H(K\cap C^-(s),L\cap C^-(s))\le\delta
\end{equation} 
implies that $z\in L$, $L\cap C^-(s_0)=\emptyset$ and $4ab\delta\le\varepsilon$. 

Let $L\in ps(C)$ be such that (\ref{4.4}) is satisfied. Then
$$ V_n((K\setminus L)\cap C^-(s)) \le \delta a,\quad V_n((L\setminus K)\cap C^-(s)) \le \delta a$$
and therefore
\begin{eqnarray*}
&& |\Ha^n_\Theta((C\setminus K)\cap C^-(s))-\Ha^n_\Theta((C\setminus L)\cap C^-(s))|\\
&& \le \left|\int_{(K\setminus L)\cap C^-(s)} \Theta(y)\,\Ha^n(\D y) \right| + \left|\int_{(L\setminus K)\cap C^-(s)} \Theta(y)\,\Ha^n(\D y) \right|\\
&& \le 2ab\delta \le\varepsilon/2.
\end{eqnarray*}
Together with (\ref{4.3}) this shows that
$$|\Ha^n_\Theta(C\setminus K)-\Ha^n_\Theta(C\setminus L)|\le \varepsilon.$$ 
There is a number $j_0$ such that for $j\ge j_0$ we have $d_H(K\cap C^-(s),K_j\cap C^-(s))\le \delta$. For $j\ge j_0$ we then have $|V_\Theta(K)-V_\Theta(K_j)|\le \varepsilon$.
\end{proof}

\section{Wulff shapes}\label{sec5}

We use Wulff shapes in given cones, as in \cite{Sch18}. Let $\emptyset\not=\omega\subset\Omega_{C^\circ}$ be a compact set and $h:\omega\to\R$ a positive continuous function. We define
$$ [h]:= C\cap\bigcap_{u\in\omega} \{y\in\R^n: \langle y,u\rangle\le -h(u)\}$$
and call this set the {\em Wulff shape} associated with $(C,\omega,h)$. It is nonempty, since $h$ is bounded. We point out that the Wulff shape associated with $(C,\omega,h)$ belongs to $\K(C,\omega)$. This follows immediately from the definition. In particular, $[h]$ is a $C$-pseudo-cone.

Relation (2.24) of \cite{HLYZ16} takes a different form. For $v\in\Omega_C$,
\begin{equation}\label{HLYZ2.24a}
\frac{1}{\varrho_{[h]}(v)}=\min_{u\in\omega} \frac{|\langle v,u\rangle|}{h(u)},
\end{equation}
where the minimum is attained. This follows from
\begin{eqnarray*}
\varrho_{[h]}(v) &=& \min\{r>0: rv\in[h]\}\\
&=& \min\{r>0: \langle rv,u\rangle \le -h(u)\mbox{ for all }u\in\omega\}\\
&=& \min\{r>0: \langle rv,u\rangle(-h(u))^{-1}\ge 1\mbox{ for all }u\in\omega\}\\
&=& \min\{r>0: r\min_{u\in\omega}|\langle v,u\rangle|h(u)^{-1}\ge 1\}\\
&=& \frac{1}{\min_{u\in\omega} |\langle v,u\rangle| h(u)^{-1}}.
\end{eqnarray*}

The following is Lemma 5 in \cite{Sch18}.

\begin{lemma}\label{L5.1}
If $h_i$, $i\in\N_0$, are positive continuous functions on $\omega$ and $h_i\to h_0$ uniformly on $\omega$, as $i\to\infty$, then $[h_i]\to [h_0]$.
\end{lemma}

Let $\omega$ and $h$ be as above, and let $h:\omega\to(0,\infty)$ and $f:\omega\to\R$ be continuous and $\delta>0$. Define $h_t$ by
$$ \log h_t(u) := \log h(u)+tf(u)+o(t,u),\quad u\in\omega,$$
for $t\in (-\delta,\delta)$, where the function $o(t,\cdot):\omega\to\R$ is continuous and satisfies $o(0,\cdot)=0$ and $\lim_{t\to 0} o(t,\cdot)/t=0$ uniformly on $\omega$. The number $\delta$ will later be chosen appropriately. Let $[h_t]$ be the Wulff shape associated with $(C,\omega,h_t)$. We call $[h_t]$ a logarithmic family of Wulff shapes associated with $(C,\omega,h,f)$.

\section{A variational lemma}\label{sec6}

Aleksandrov's variational lemma, with its original proof, is reproduced as Lemma 7.5.3 in \cite{Sch14}. This is an essential part of the classical approach to Minkowski's existence theorem. Its proof makes use of Minkowski's inequalities for mixed volumes. It was carried over to pseudo-cones in \cite{Sch18}, but the method does not permit other desired generalizations. Fortunately, Huang, Lutwak, Yang and Zhang \cite{HLYZ16} found a way to prove this lemma in a different way, and this approach also allowed to treat weighted surface area measures of convex bodies, see Kryvonos and Langharst \cite{KL23}. We need a counterpart of this approach for pseudo-cones. In the following, we adapt the proofs of \cite{HLYZ16} and \cite{KL23} to this case. 

First we carry over Lemma 4.3 of \cite{HLYZ16} to pseudo-cones. We emphasize that the main ideas are taken from \cite{HLYZ16}, but our proof would be different even for convex bodies, since we avoid the use of convexifications and polarity. In \cite{HLYZ16} these are employed to obtain a whole series of variational lemmas.

In the following, we restrict ourselves to compact sets $\omega\subset\Omega_{C^\circ}$ and then to $C$-pseudo-cones that are determined by $\omega$. The reason is that we need some compactness arguments. Let $K\in\K(C,\omega)$, and recall (from Definition \ref{D2.3}) that $\eta_K$ denotes the set of all $v\in\Omega_C$ for which the outer unit normal vector of $K$ at $\varrho_{K}(v)v$ is not unique. If $v\in\Omega_C\setminus\eta_K$, then the supporting hyperplane of $K$ at $\varrho_K(v)v$ is unique, hence its outer unit normal vector belongs to $\omega$. In particular, the radial Gauss map $\alpha_K$ of $K\in\K(C,\omega)$ maps $\Omega_C\setminus\eta_K$ into $\omega$. This will be important.

\begin{lemma}\label{HLYZL4.3}
Let $\omega\subset \Omega_{C^\circ}$ be nonempty and compact, and let $h_0:\omega\to (0,\infty)$ and $f:\omega\to \R$ be continuous. If $[h_t]$ is a logarithmic family of Wulff shapes associated with $(C,\omega,h_0,f)$, where
$$ \log h_t(u) = \log h_0(u)+tf(u)+o(t,u)$$
for $u\in\omega$, then for almost all $v\in\Omega_C$,
\begin{equation}\label{V0} 
\lim_{t\to 0} \frac{\log \varrho_{[h_t]}(v) -\log\varrho_{[h_0]}(v)}{t}= f(\alpha_{[h_0]}(v)).
\end{equation}
\end{lemma}

\begin{proof}
We fix a vector $v\in\Omega_C \setminus \eta_{[h_0]}$. The set $\eta_{[h_0]}$ has spherical Lebesgue measure zero.

By formula (\ref{HLYZ2.24a}) (where the minimum is attained if $v\in\Omega_C$), there exists $u_t\in\omega$ such that
\begin{equation}\label{V2} 
\frac{1}{\varrho_{[h_t]}(v)} = \frac{|\langle v,u_t\rangle|}{h_t(u_t)} \quad\mbox{and} \quad \frac{1}{\varrho_{[h_t]}(v)} \le \frac{|\langle v,u\rangle|}{h_t(u)} \quad\mbox{for all } u\in\omega.
\end{equation}
By (\ref{V2}), for $t=0$,
$$ |\langle \varrho_{[h_0]}(v)v,u_0\rangle|= h_0(u_0).$$
Since $[h_0]$ is the Wulff shape associated with $(C,\omega,h_0)$, we have $h_0(u_0)\le \overline h_{[h_0]}(u_0)$, hence
$$ |\langle \varrho_{[h_0]}(v)v,u_0\rangle| \le \overline h_{[h_0]}(u_0) = \min_{x\in\partial[h_0]} |\langle x,u_0\rangle|.$$
From $\varrho_{[h_0]}(v)v\in\partial [h_0]$ it follows that the equality sign holds here, hence $u_0$ is an (in fact the unique) outer unit normal vector of $[h_0]$ at $\varrho_{[h_0]}(v)v$, thus
\begin{equation}\label{V4}
u_0 = \alpha_{[h_0]}(v).
\end{equation} 

We state that
\begin{equation}\label{V5}
\lim_{t\to 0} u_t=u_0.
\end{equation}
We argue as in \cite{HLYZ16} (proof of (4.8)), with the necessary changes. Since $\omega$ is compact, there exists a sequence $(u_{t_k})_{k\in\N}$ with $t_k\to 0$ that converges to some $u'\in\omega$. It suffices to prove that every convergent sequence $(u_{t_k})_{k\in\N}$ with $t_k\to 0$ converges to $u_0$. Suppose that $t_k\to 0$ as $k\to\infty$ and
$$ \lim_{k\to\infty} u_{t_k} = u'\in\omega.$$
Since $h_{t_k}\to h_0$ uniformly on $\omega$ and since $h_0>0$ on $\omega$, we get
\begin{equation}\label{V6}
\frac{1}{\varrho_{[h_{t_k}]}(v)} = \frac{|\langle v,u_{t_k}\rangle|}{h_{t_k}(u_{t_k})} \to \frac{|\langle v, u'\rangle|}{h_0(u')}.
\end{equation}
From Lemma \ref{L5.1} it follows that also
$$ \frac{1}{\varrho_{[h_{t_k}]}(v)} \to \frac{1}{\varrho_{[h_o]}(v)},$$
hence
$$ \frac{1}{\varrho_{[h_o]}(v)} = \frac{|\langle v, u'\rangle|}{h_0(u')}.$$
As noted above, we have $h_0(u)\le\overline h_{[h_0]}(u)$, hence
$$ \frac{1}{\varrho_{[h_o]}(v)} \ge \frac{|\langle v, u'\rangle|}{\overline h_{[h_0]}(u')} \ge\frac{1}{\varrho_{[h_o]}(v)},$$
the latter by (\ref{2.0}). Since here the equality sign holds, the remark after (\ref{2.0}) shows that $u'$ is an outer unit normal vector of $[h_0]$ at $\varrho_{[h_0]}(v)v$ and hence coincides with $u_0$. This proves (\ref{V5}).

By (\ref{2.0}) we have
\begin{equation}\label{V7}
\frac{1}{\varrho_{[h_0]}(v)} \le \frac{|\langle v, u_t\rangle|}{\overline h_{[h_0]}(u_t)}.
\end{equation}

Using (\ref{V2}), (\ref{V7}), the inequality $\overline h_{[h_0]}\ge h_0$, and the definition of $h_t$, we obtain
\begin{eqnarray}
\log \varrho_{[h_t]}(v) - \log \varrho_{[h_0]}(v) &=& \log h_t(u_t) -\log |\langle v,u_t\rangle| -\log \varrho_{[h_0]}(v)\nonumber\\
&\le & \log h_t(u_t) -\log\overline h_{[h_0]}(u_t)\nonumber\\
&\le& \log h_t(u_t) - \log h_0(u_t)\nonumber\\
&=& tf(u_t) + o(t,u_t).\label{V8}
\end{eqnarray}

Using (\ref{V2}) (for $t=0$) and then (\ref{HLYZ2.24a}), we get 
\begin{eqnarray}
\log \varrho_{[h_t]}(v) - \log \varrho_{[h_0]}(v) &=& \log \varrho_{[h_t]}(v) -\log h_0(u_0) + \log |\langle v,u_0\rangle|\nonumber\\
&\ge & \log h_t(u_0) - \log h_0(u_0)\nonumber\\
&=& tf(u_0)+o(t,u_0).\label{V9}
\end{eqnarray}

From (\ref{V8}) and (\ref{V9}) we see that
\begin{equation}\label{V10}
 tf(u_0)+o(t,u_0) \le \log \varrho_{[h_t]}(v) - \log \varrho_{[h_0]}(v) \le tf(u_t) + o(t,u_t).
\end{equation}
These inequalities, together with the continuity of the function $f$ and the limit relation (\ref{V5}), yield
$$ \lim_{t\to 0} \frac{\log \varrho_{[h_t]}(v) - \log \varrho_{[h_0]}(v)}{t} = f(u_0) = f(\alpha_{[h_0]}(v))$$
and thus the assertion (\ref{V0}).
\end{proof}

We adapt Lemma 2.6 of \cite{KL23} to $C$-determined pseudo-cones. 

\begin{lemma}\label{KLL2.6}
Let $\omega\subset\Omega_{C^\circ}$ be nonempty and compact, let $K\in\K(C,\omega)$. Let $f:\omega\to \R$ be continuous. There is a  constant $\delta> 0$ such that the function $h_t$ defined by
$$ h_t(u):= \overline h_K(u) +tf(u), \quad u\in\omega,$$
is positive for $|t|\le \delta$. Let $[h_t]$ be the Wulff shape associated with $(C,\omega,h_t)$, for $|t|\le\delta$.\\[1mm]
\noindent $\rm (a)$ For almost all $v\in\Omega_C$,
$$ \frac{\D \varrho_{[h_t]}(v)}{\D t} \Big|_{t=0} = \lim_{t\to 0} \frac{\varrho_{[h_t]}(v) - \varrho_K(v)}{t} = \frac{f(\alpha_K(v))}{\overline h_K(\alpha_K(v))}\varrho_K(v).$$
\noindent $\rm (b)$ There is a constant $M$ with
$$ |\varrho_{[h_t]}(v) - \varrho_K(v)|\le M|t|$$
for all $v\in\Omega_C$ and all $|t|\le \delta$.
\end{lemma}

\begin{proof}
Since  $K\in\K(C,\omega)$, the absolute support function of $K$ is bounded away from $0$ on $\omega$, hence we can choose $\delta_1>0$ such that $h_t(u)>0$ for all $u\in\omega$ if $|t|\le \delta_1$. From the Taylor expansion of $\log(1+x)$ we get
$$ \log h_t = \log \overline h_K +\frac{tf}{\overline h_K} + o(t,\cdot)\quad\mbox{on }\omega,$$
so that $[h_t]$ is a logarithmic family of Wulff shapes associated with $(C,\omega,\overline h_K|_\omega,f/\overline h_K|_\omega)$. Therefore, Lemma \ref{HLYZL4.3} can be applied, which yields (with $h_0=\overline h_K|_\omega$ and hence $[h_0]=K$)
$$ \frac{\D \log \varrho_{[h_t]}(v)}{\D t}\Big |_{t=0} = \frac{f(\alpha_{[h_0]}(v))}{\overline h_K(\alpha_{[h_0]}(v))}= \frac{f(\alpha_K(v))}{\overline h_K(\alpha_K(v))}$$
for almost all $v\in\Omega_C$. From
$$ \frac{\D \log \varrho_{[h_t]}(v)}{\D t}\Big |_{t=0} = \frac{1}{\varrho_{[h_0]}(v)} \frac{\D \varrho_{[h_t]}(v)}{\D t}\Big |_{t=0}$$
we get
$$ \frac{\D \varrho_{[h_t]}(v)}{\D t}\Big |_{t=0} =  \frac{f(\alpha_K(v))}{\overline h_K(\alpha_K(v))} \varrho_K(v),$$
which is assertion (a), if $\delta\le \delta_1$.

From (a) we further get, for almost all $v\in\Omega_C$,
$$ \Big |\frac{\varrho_{[h_t]}(v) - \varrho_K(v)}{t}\Big | - \Big |\frac{\D \varrho_{[h_t]}(v)}{\D t} \Big|_{t=0}  \le 
\Big |\frac{\varrho_{[h_t]}(v) - \varrho_K(v)}{t}  -\frac{\D \varrho_{[h_t]}(v)}{\D t} \Big|_{t=0} \Big |\le 1,$$
say, for $|t| \le \delta_2$, if $\delta_2>0$ is suitably chosen. This gives
$$ \Big |\frac{\varrho_{[h_t]}(v) - \varrho_K(v)}{t}\Big | \le  \Big|\frac{\D \varrho_{[h_t]}(v)}{\D t} \Big|_{t=0}+1 = \varrho_K(v) \Big | \frac{\D \log \varrho_{[h_t]}(v)}{\D t}\Big|_{t=0} +1.$$
Here $\varrho_K$ is bounded since $K\in\K(C,\omega)$, and by Lemma \ref{HLYZL4.3}, 
$$ \Big | \frac{\D \log \varrho_{[h_t]}(v)}{\D t}\Big|_{t=0} = \frac{f}{\overline h_K|_\omega}(\alpha_K(v)),$$
where $f/\overline h_K|_\omega$ is bounded. This, together with the continuity of radial functions, proves assertion (b), for $\delta\le \delta_2$.
\end{proof}

\vspace{3mm}

Let $\Theta$ be according to Definition \ref{D1.1}, with $n-1<q<n$. For a $C$-pseudo-cone $K$, we recall that the measure $S_{n-1}^\Theta(K,\cdot)$ on $\Omega_{C^\circ}$ is defined as the image measure of the measure given by
$$ \beta\mapsto  \int_\beta \Theta(y) \,\Ha^{n-1}(\D y), \quad \beta\in\B(\partial_i K),$$
under the spherical image map, and that
$$ V_\Theta(K):= \int_{C\setminus K} \Theta(y) \,\Ha^n(\D y).$$
As shown in Lemmas \ref{L3.1} and \ref{L4.1}, the weighted surface area measure $S^\Theta_{n-1}(K,\cdot)$ and the weighted covolume $V_\Theta$ are finite.

The following adapts Lemma 2.7 of \cite{KL23} to $C$-full pseudo-cones (but we need only special densities $\Theta$ as considered here). 

\begin{lemma}\label{KL23L2.7}
Let $K\in\K(C,\omega)$, for some nonempty, compact set $\omega\subset\Omega_{C^\circ}$. Let $\Theta$, $S^\Theta_{n-1}(K,\cdot)$ and $V_\Theta$ be as above. Let $f:\omega\to\R$ be continuous, and let $[\overline h_K|_\omega+tf]$ be the Wulff shape associated with $(C,\omega, \overline h_K|_\omega+tf)$.  Then
$$ \lim_{t\to 0} \frac{V_\Theta([\overline h_K|_\omega+tf])-V_\Theta(K)}{t} = \int_\omega f(u)\,S_{n-1}^\Theta(K,\D u).$$
\end{lemma}

\begin{proof}
We write $h_t(u):= \overline h_K(u)+tf(u)$ for $u\in\omega$ and small enough $|t|$ and use spherical polar coordinates to obtain
$$ V_\Theta([h_t]) = \int_{\Omega_C}\int_0^{\varrho_{[h_t]}(v)} \Theta(rv) r^{n-1}\,\D r \, \D v = \int_{\Omega_C} F_t(v)\,\D v,$$
where
$$ F_t(v):= \int_0^{\varrho_{[h_t]}(v)} \Theta(rv)r^{n-1}\,\D r$$
for $v\in\Omega_C$. This gives
\begin{eqnarray}
\frac{F_t(v)-F_0(v)}{t} &=& \frac{1}{t} \int_{\varrho_K(v)}^{\varrho_{[h_t]}(v)} \Theta(rv)r^{n-1}\,\D r\label{7.2a}\\
&=& \frac{\varrho_{[h_t]}(v)-\varrho_K(v)}{t}\cdot\frac{1}{\varrho_{[h_t]}(v)-\varrho_K(v)} \int_{\varrho_K(v)}^{\varrho_{[h_t]}(v)} \Theta(rv)r^{n-1}\,\D r.\nonumber
\end{eqnarray}
As $t\to 0$, the first factor converges, by Lemma \ref{KLL2.6}, to $\varrho_K(v) f(\alpha_K(v))/\overline h_K(\alpha_K(v))$, for almost all $v\in\Omega_C$. The second factor converges to $\Theta(\varrho_K(v)v)\varrho^{n-1}_K(v)=\Theta(r_K(v))\varrho^{n-1}_K(v)$. Thus, we get
$$ \lim_{t\to 0} \frac{F_t(v)-F_0(v)}{t} = \Theta(r_K(v)) \frac{f(\alpha_K(v))}{\overline h_K(\alpha_K(v))}\varrho^n_K(v)$$
for almost all $v\in\Omega_C$.

By (\ref{7.2a}) and Lemma \ref{KLL2.6} we have
$$ \left|\frac{F_t(v)-F_0(v)}{t}\right| \le M'$$
with some constant $M'$ independent of $t$ (for sufficiently small $|t|$), hence the dominated convergence theorem  allows us to conclude that
\begin{eqnarray*}
\lim_{t\to 0} \frac{V_\Theta([h_t]) - V_\Theta(K)}{t} &=& \int_{\Omega_C} \lim_{t\to 0} \frac{F_t(v)-F_0(v)}{t}\,\D v\\
&=& \int_{\Omega_C} \Theta(r_K(v))\frac{f(\nu_K(r_K(v))}{\overline h_K(\alpha_K(v))}\varrho_K^n(v)\,\D v\\
&=& \int_{\partial_i K}f(\nu_K(y))\Theta(y)\,\Ha^{n-1}(\D y)\\
&=& \int_\omega f(u)\, S_{n-1}^\Theta(K,\D u).
\end{eqnarray*}
Here we have used Lemma \ref{L2a.1} and the definition of $S_{n-1}^\Theta(K,\cdot)$, together with the transformation formula for integrals. The final integral is only over $\omega$, since 
$$\Ha^{n-1}(\nu_K^{-1}(\Omega_C\setminus \omega))=0$$
for the set $K$ determined by $\omega$.
\end{proof}

\section{Proof of Theorem \ref{T1.1}}\label{sec7}

We assume that $\Theta$ is chosen according to Definition \ref{D1.1}, with $n-1<q<n$. Under this assumption, we have already proved in Lemmas \ref{L3.1} and \ref{L4.1} that $S_{n-1}^\Theta(K,\cdot)$ and $V_\Theta(K)$ are finite for every $C$-pseudo-cone $K$.

First we state two lemmas, which will later be needed.

\begin{lemma}\label{L7.1}
There is a constant $c$, depending only on $C$ and $\Theta$, such that every pseudo-cone $K\in ps(C)$ with $V_\Theta(K)=1$ satisfies $\overline h_{K}\le c$.
\end{lemma}

\begin{proof}
Let $r$ be the maximal radius of a ball with center $o$ having empty intersection with ${\rm int}\,K$. Since $rB^n\cap K\not=\emptyset$, every supporting hyperplane of $K$ with an outer normal vector $u\in\Omega_{C^\circ}$ intersects $rB^n$, hence $\overline h_{K}(u)\le r$. From $V_\Theta(K)=1$ it follows that $\Ha^n_\Theta(C\cap rB^n)<1$. Since $\Ha^n_\Theta$ is homogeneous of degree $n-q>0$, we get $1> \Ha^n_\Theta(C\cap rB^n) = r^{n-q}\Ha^n_\Theta(C\cap B^n)$, which gives the desired bound.
\end{proof}

For a pseudo-cone $K\in ps(C)$ and a nonempty compact set $\beta\subset\Omega_{C^\circ}$, we recall from (\ref{2.1}) that the set $K^{(\beta)}\in\K(C,\beta)$ was defined by
$$ K^{(\beta)}:= C\cap\bigcap_{u\in\beta} H^-(K,u).$$

\begin{lemma}\label{L7.2}
Let $K\in ps(C)$ and let $\omega,\beta\subset\Omega_{C^\circ}$ be nonempty compact sets with $\omega\subset{\rm int}\,\beta$. Then $\bm x_K(\omega)= \bm x_{K^{(\beta)}}(\omega)$.
\end{lemma}

\begin{proof}
The inclusion $\bm x_K(\omega)\subseteq \bm x_{K^{(\beta)}}(\omega)$ holds trivially. Suppose that $y\in\bm x_{K^{(\beta)}}(\omega)$, but $y\notin \bm x_K(\omega)$. Then $y\in H(K,u)$ for some $u\in\omega$, but $y\notin K$. There are unit vectors $u'$, arbitrarily close to $u$, such that some hyperplane with normal vector $u'$ strictly separates $y$ from $K$. Since $u\in\omega\subset{\rm int}\,\beta$, the vector $u'$ can be chosen in $\beta$, and then $y\notin K^{(\beta)}$ and thus $y\notin \bm x_{K^{(\beta)}}$, a contradiction.
\end{proof}

For the proof of Theorem 1, we can use a similar argumentation as in \cite{Sch18, Sch21}, but with modifications. First we assume that $\omega\subset\Omega_{C^\circ}$ is nonempty and compact. Let $\varphi$ be a nonzero, finite Borel measure on $\omega$. By $\C^+(\omega)$ we denote the set of continuous functions $f:\omega\to(0,\infty)$, with the topology induced by the maximum norm. For $f\in\C^+(\omega)$ we recall that $[f]$ denotes the Wulff shape associated with $(C,\omega,f)$. Then we define a functional $\Phi:\C^+(\omega) \to (0,\infty)$ by
$$ \Phi(f):= V_\Theta( [f])^{-1/(n-q)}\int_\omega f\,\D\varphi.$$
It follows from Lemmas \ref{L5.1} and \ref{L4.1} that the functional $\Phi$ is continuous. We show that it attains a maximum on the set
$$ \mathcal L':= \{\overline h_L|_\omega: L\in\K(C,\omega),\, V_\Theta( L)=1\}.$$
The condition $V_\Theta( L)=1$ can be satisfied, because  $L\in\K(C,\omega)$ implies that $ L\subset C^-(t)$ for some $t$ and then (\ref{4.2}) and the homogeneity of $V_\Theta$ can be applied.
By an extension of \cite[Lem. 8]{Sch18}, there is a number $t>0$ such that
$$ K\in\K(C,\omega) \wedge V_\Theta( K)=1 \Rightarrow C(t)\subset K.$$
In fact, the proof of that Lemma extends to $V_\Theta$, because (\ref{4.2}) holds. If now $L\in \K(C,\omega)$ is such that $\overline h_L|_\omega\in\mathcal L'$, then with $t$ as above (independent of $L$) we have
$$ \Phi(\overline h_L|_\omega) \le \int_\omega -h_{C(t)}\,\D \varphi <c$$
with a constant $c$ independent of $L$. Thus, $\sup\{\Phi(f):f\in\mathcal L'\}<\infty$. Precisely as in \cite{Sch18} (Proof of Theorem 3) one can now show that there exists a $C$-pseudo-cone $K_0$ such that $\overline h_{K_0}|_\omega\in\mathcal L'$ and $\Phi$ attains its maximum on $\mathcal L'$ at $\overline h_{K_0}|_\omega$ (the condition $o\notin K_0$ is obviously satisfied, since $V_\Theta(K_0)=1$).

Since $V_\Theta$ is homogeneous of degree $n-q$, the functional $\Phi$ is homogeneous of degree zero. It follows that the maximum of $\Phi$ on the set $\mathcal L:= \{\overline h_L|_\omega:L\in\K(C,\omega)\}$ is also attained at $\overline h_{K_0}|_\omega$.

For $K\in\K(C,\omega)$ it follows from the definition of the Wulff shape that $[\overline h_K|_\omega]=K$. Let $f\in \C^+(\omega)$. Since $[f]\in\K(C,\omega)$, we have $ [\overline h_{[f]}|_\omega]=[f]$, hence $V_\Theta([f]) = V_\Theta([\overline h_{[f]}|_\omega])$. Further, $f(u)\le \overline h_{[f]}(u)$ for $u\in\omega$. Therefore,
\begin{eqnarray*} 
\Phi(f) &=& V_\Theta([f])^{-1/(n-q)}\int_\omega f\,\D\varphi \le V_\Theta([\overline h_{[f]}|_\omega])^{-1/(n-q)}\int_\omega \overline h_{[f]}\,\D\varphi\\ 
&=& \Phi(\overline h_{[f]}|_\omega)\le \Phi(\overline h_{K_0}|_\omega).
\end{eqnarray*}
Thus, the maximum of $\Phi$ on $\C^+(\omega)$ is also attained at $\overline h_{K_0}|_\omega$.

Now let $f:\omega\to\R$ be continuous. For sufficiently small $|\tau|$, the function $\tau\mapsto \overline h_{K_0}|_\omega+\tau f$ belongs to $\C^+(\omega)$, hence the function
defined by $F(\tau):= \Phi(\overline h_{K_0}|_\omega+\tau f)$ has a maximum at $\tau=0$. By Lemma \ref{KL23L2.7} (and since we have $V_\Theta( K_0)=1$), we get
$$ F'(0) =-\frac{1}{n-q}\int_\omega f\,\D S_{n-1}^\Theta(K_0,\cdot)\cdot\int_\omega \overline h_{K_0}\,\D\varphi + \int_\omega f\,\D\varphi.$$
Since this is equal to $0$, setting
$$ \lambda:= \frac{1}{n-q} \int_\omega\overline h_{K_0}\,\D\varphi$$
we obtain
$$ \int_\omega f\,\D\varphi =\lambda \int_\omega f\,\D S_{n-1}^\Theta(K_0,\cdot).$$
Since this holds for all continuous functions $f$ on $\omega$, we deduce that
$$ \varphi = \lambda S_{n-1}^\Theta(K_0,\cdot) = S_{n-1}^\Theta(K,\cdot)$$
with $K:= \lambda^{\frac{1}{n-1-q}}K_0$, since $S_{n-1}^\Theta$  is homogeneous of degree $n-1-q$ in its first argument.

Before we continue, we introduce some notation. For a $C$-pseudo-cone $K$, we denote by $b(K)$ the distance of $K$ from the origin. For $u\in\Omega_{C^\circ}$, we denote by $\delta_C(u)$ the spherical distance of $u$ from the boundary of $\Omega_{C^\circ}$, and for $\alpha>0$ we write
$$ \omega(\alpha) := \{u\in\Omega_{C^\circ}: \delta_C(u)\ge\alpha\}.$$
This set is compact (possibly empty).

Now we assume, as in the formulation of Theorem \ref{T1.1}, that $\varphi$ is a nonzero, finite Borel measure on all of $\Omega_{C^\circ}$. We choose a number $\tau>0$ such that $\varphi(\omega(\tau))>0$. Then we choose a sequence $(\omega_j)_{j\in\N}$ of compact subsets of $\Omega_{C^\circ}$ with $\omega_1:= \omega(\tau)$, $\omega_j\subset {\rm int}\,\omega_{j+1}$ for $j\in\N$, and $\bigcup_{j\in\N}\omega_j=\Omega_{C^\circ}$. To each $j\in \N$ we define a measure $\varphi_j$ by $\varphi_j(\omega) := \varphi(\omega\cap \omega_j)$ for $\omega\in\B(\Omega_{C^\circ})$. Then for each $j\in \N$ there exists, as shown above, a $C$-pseudo-cone $K_j\in\K(C,\omega_j)$ satisfying 
$$ V_\Theta(K_j) =1\quad\mbox{and}\quad \varphi_j=\lambda_j S^\Theta_{n-1}(K_j,\cdot)$$
with
$$ \lambda_j := \frac{1}{n-q} \int_{\omega_j} \overline h_{K_j}\,\D\varphi_j = \frac{1}{n-q} \int_{\omega_j} \overline h_{K_j}\,\D\varphi,$$
so that the pseudo-cone
$$ L_j:= \lambda_j^{\frac{1}{n-1-q}}K_j\quad\mbox{satisfies}\quad S^\Theta_{n-1}(L_j,\cdot)=\varphi_j.$$

We have the intention to apply Lemma \ref{L2.1} to the sequence $(L_j)_{j\in\N}$, hence we must know that the distances $b(L_j)$ from the origin are bounded and bounded away from $0$. First we note that for the sequence $(K_j)_{j\in\N}$ it follows from Lemma \ref{L7.1} that $\overline h_{K_j}\le c$ with a constant $c$ independent of $j$. Since
$$\lambda_j\le \frac{1}{n-q} \int_{\omega_j} c\,\D\varphi \le \frac{c}{n-q}\varphi(\Omega_{C^\circ})<\infty,$$
it follows that
$$ \overline h_{L_j} \le c\cdot\left(\frac{c}{n-q}\varphi(\Omega_{C^\circ})\right)^{\frac{1}{n-1-q}}.$$
Hence, $(b(L_j))_{j\in\N}$ is bounded from above, and we can choose a number $t_1>0$ such that
$$ L_j\cap C^-(t_1) \not=\emptyset\quad\mbox{for all }j\in\N.$$
Further, we have $S_{n-1}(L_j,\omega(\tau))=\varphi(\omega(\tau)) =:s>0$, and it follows from \cite[Lem. 9]{Sch23} that $b(L_j)>b_0$ with some constant $b_0>0$ depending only on $C,\tau$ and $s$.

Now we can complete the proof essentially as in \cite{Sch21}, Proof of Theorem 1. As there, we choose (with $t_1$ as above) a real sequence $(t_k)_{k\in \N}$ with $t_k\uparrow \infty$ as $k\to\infty$. Then there exist a $C$-pseudo-cone $K$ (that $K$ is a $C$-pseudo-cone, is clear from the construction) and a subsequence $(\ell_i)_{i\in\N}$ of $\N$ such that
$$ \lim_{i\to\infty}(L_{\ell_i} \cap C^-(t_k))= K\cap C^-(t_k)\quad\mbox{for each }k\in\N.$$

We change the notation, replacing $\omega_{\ell_i}$ by $\omega_i$, $\varphi_{\ell_i}$ by $\varphi_i$, and $L_{\ell_i}$ by $L_i$, then
\begin{equation}\label{7.2}
\lim_{i\to\infty}(L_i \cap C^-(t_k))= K\cap C^-(t_k)\quad\mbox{for each }k\in\N.
\end{equation}
We fix $i$, choose a compact set $\beta\subset\Omega_{C^\circ}$ with $\omega_i\subset{\rm int}\,\beta$, and choose $k$ so large that $K^{(\beta)} \subset C^-(t_k)$. Then it follows from Lemma \ref{L7.2} that
$$ \bm x_K(\omega_i)= \bm x_{K^{(\beta)}}(\omega_i)$$
and from Lemma \ref{L2.2} that
$$ \lim_{i\to \infty} (L_i^{(\beta)}\cap C^-(t_k)) = K^{(\beta)}\cap C^-(t_k)).$$
Lemma \ref{L3.3} then shows that the restriction of $S^\Theta_{n-1}(L_i,\cdot)$ to $\omega_i$ converges weakly to the restriction of $S^\Theta_{n-1}(K^{(\beta)},\cdot)$ to $\omega_i$. Since the restriction of $S^\Theta_{n-1}(L_i,\cdot)$ to $\omega_i$ is equal to the restriction of $\varphi$ to $\omega_i$, it follows that for each Borel set $\omega'\subset \omega_i$ we have $S^\Theta_{n-1}(K,\omega') = S^\Theta_{n-1}(K^{(\beta)},\omega')= \varphi_i(\omega')= \varphi(\omega')$. 

Since $\bigcup_{i\in\N} \omega_i=\Omega_{C^\circ}$, we deduce that $S^\Theta_{n-1}(K,\cdot)=\varphi$. This completes the proof of Theorem \ref{T1.1}. \hfill$\Box$

\vspace{3mm}

We add a remark on the failing uniqueness, extending an observation made in \cite[Sect. 6]{Sch23}. Suppose that the measure $\varphi$ is concentrated in the one-pointed set $\omega:=\{\mathfrak{-v}\}$ and assigns the value $1$ to it. Since the function $\vartheta(t):= \int_{C(t)} \Theta(y)\,\Ha^{n-1}(\D y)$, $t>0$, is continuous and homogeneous of degree $n-1-q<0$, there is a number $t_0$ with $\vartheta(t_0)=1$. Then the pseudo-cone $C(t_0)+C$ is in $\K(C,\omega)$ and satisfies $S^\Theta_{n-1}(K,\cdot)=\varphi$. For arbitrary $t_1\in (0,t_0)$ we have $\vartheta(t_1)>\vartheta(t_0)=1$ and hence can choose an $(n-1)$-dimensional closed convex set $F\subset C(t_1)$ which satisfies $\int_F\Theta(y)\,\Ha^{n-1}(\D y) =1$ and is not a translate of $C(t_0)$. Then $L:= F+C$ is a pseudo-cone for which $S^\Theta_{n-1}(L,\cdot)$ is equal to $S^\Theta_{n-1}(K,\cdot)$ and has support contained in $\omega$, but $K$ is not a translate of $L$.

It remains open whether uniqueness holds if both, $K$ and $L$, belong to $\K(C,\omega)$.

\noindent Author's address:\\[2mm]
Rolf Schneider\\Mathematisches Institut, Albert--Ludwigs-Universit{\"a}t\\D-79104 Freiburg i.~Br., Germany\\E-mail: rolf.schneider@math.uni-freiburg.de

\end{document}